\documentclass[12pt]{article}
\usepackage{amsmath,amsthm}
\usepackage{amssymb}
\newtheorem{theorem}{Theorem}[section]
\newtheorem{proposition}[theorem]{Proposition}
\newtheorem{lemma}[theorem]{Lemma}
\newtheorem{corollary}[theorem]{Corollary}
\newtheorem{conjecture}[theorem]{Conjecture}
\newenvironment{remark}{\noindent{\bf Remark.}}{}
\newenvironment{remarks}{\noindent{\bf Remarks.}}{}
\newenvironment{example}{\noindent{\bf Example.}}{}
\def\nolineskip{\baselineskip=0pt \lineskip=0pt}
\def\orb{\mathop{\hbox{\rm Orb}}\nolimits}
\def\corb{\mathop{\overline{\hbox{\rm Orb}}}\nolimits}
\def\lfa#1#2{\mathord{\hbox{\vtop{\kern-6pt\hbox{$#2$}}$\setminus$$#1$}}}
\def\dvd{\mathrel{
\vcenter{\nolineskip\hbox{.}\kern2pt\hbox{.}\kern2pt\hbox{.}}}}
\def\notdvd{\mathrel{\hbox to 0pt{\kern1pt
$\vcenter{\nolineskip\hbox{.}\kern2pt\hbox{.}\kern2pt\hbox{.}}$\hss}/}}
\def\rest#1{\raise-2pt\hbox{$|_{#1}$}}
\def\Rspan{\mathop{\hbox{\rm span}_{\R}}}
\def\Zspan{\mathop{\hbox{\rm span}_{\Z}}}
\def\Clim{\mathop{\hbox{\rm UC-lim}}}
\let\scr=\scriptstyle
\let\txt=\textstyle

\let\ras=\rightarrow
\let\ra=\longrightarrow
\let\sle=\subseteq
\let\sln=\subset
\let\sge=\supseteq
\let\overline=\overline
\let\tens=\otimes
\let\sm=\setminus
\let\col=\colon
\let\ld=\ldots
\let\cd=\cdot
\let\eps=\varepsilon

\def\comp{\mathord{\scr\circ}}
\def\nasp{\spacefactor=1000}
\def\smatr#1{{\baselineskip=2pt\lineskip=2pt
\left(\vcenter{\let\\=\cr\let\-=\hfill
\ialign{\hfil$\scr##$\hfil&&\hfil\kern2pt$\scr##$\hfil\cr#1\crcr}}\right)}}
\def\svect#1{\smatr{#1}}
\def\svd{\vbox to 2.4mm{\nolineskip
\kern.3pt\hbox{.}\vfil\hbox{.}\vfil\hbox{.}\kern.3pt}}
\def\vd{\vbox to 3.2mm{\nolineskip
\kern1pt\hbox{.}\vfil\hbox{.}\vfil\hbox{.}\kern1pt}}
\def\N{{\mathbb N}}
\def\Z{{\mathbb Z}}
\def\Q{{\mathbb Q}}
\def\R{{\mathbb R}}
\def\T{M}
\let\alf=\alpha
\let\gam=\gamma
\let\lam=\lambda
\let\del=\delta
\let\Del=\Delta
\let\Gam=\Gamma
\let\sig=\sigma
\let\lam=\lambda
\let\Lam=\Lambda
\def\tmu{\widetilde{\mu}}
\def\B{{\cal B}}
\def\bx{\bar{x}}
\def\Gc{G^{o}}
\def\Xc{X^{o}}
\def\hG{\hat{G}}

\def\hq{\hat{q}}

\def\tX{\widetilde{X}}
\def\tT{\widetilde{T}}
\def\tS{\widetilde{S}}
\begin{document}
\title{Intersective polynomials\\
and polynomial Szemer\'{e}di theorem}
\author{V. Bergelson$^{*}$, A. Leibman\thanks{The first two authors were supported 
by NSF grant DMS-0600042}\ ,
and E. Lesigne}
\maketitle
\begin{abstract}
Let $P=\{p_{1},\ld,p_{r}\}\subset\Q[n_{1},\ld,n_{m}]$ 
be a family of polynomials such that $p_{i}(\Z^{m})\sle\Z$, $i=1,\ld,r$.
We say that the family $P$ has {\it PSZ property\/}
if for any set $E\sle\Z$ with 
$d^{*}(E)=\limsup_{N-M\ras\infty}\frac{|E\cap[M,N-1]|}{N-M}>0$
there exist infinitely many $n\in\Z^{m}$
such that $E$ contains a polynomial progression of the form
\hbox{$\{a,a+p_{1}(n),\ld,a+p_{r}(n)\}$}.
We prove that a polynomial family $P=\{p_{1},\ld,p_{r}\}$ has PSZ property
if and only if the polynomials $p_{1},\ld,p_{r}$ are {\it jointly intersective\/},
meaning that for any $k\in\N$ there exists $n\in\Z^{m}$
such that the integers $p_{1}(n),\ld,p_{r}(n)$ are all divisible by $k$.
To obtain this result 
we give a new ergodic proof of the polynomial Szemer\'{e}di theorem,
based on the fact that the key to the phenomenon of polynomial multiple recurrence
lies with the dynamical systems defined by translations on nilmanifolds.
We also obtain, as a corollary, the following generalization 
of the polynomial van der Waerden theorem:
If $p_{1},\ld,p_{r}\in\Q[n]$ are jointly intersective integral polynomials,
then for any finite partition of\/ $\Z$, $\Z=\bigcup_{i=1}^{k}E_{i}$,
there exist $i\in\{1,\ld,k\}$ and $a,n\in E_{i}$
such that $\{a,a+p_{1}(n),\ld,a+p_{r}(n)\}\sln E_{i}$.
\end{abstract}
%=====================================================SSSSSSSSS
\section{Introduction}
\label{S-Intro}
%----------------------
Let us call a polynomial $p\in\Q[n]$ {\it integral\/} 
if it takes on integer values on the integers.
The polynomial Szemer\'{e}di theorem (\cite{psz})
states that if a set $E\sle\Z$ has positive upper Banach density,
$d^{*}(E)=\limsup_{N-M\ras\infty}\frac{|E\cap[M,N-1]|}{N-M}>0$,
then for any finite family of integral polynomials $P=\{p_{1},\ld,p_{r}\}$
with $p_{i}(0)=0$, $i=1,\ld,r$,
one can find an arbitrarily large $n\in\N$ such that, for some $a\in E$,
$\{a,a+p_{1}(n),\ld,a+p_{r}(n)\}\sln E$.
Moreover, the set 
$$
N_{P}(E)=\Bigl\{n\in\Z:\hbox{for some $a$, 
$\{a,a+p_{1}(n),\ld,a+p_{r}(n)\}\sln E$}\Bigr\}
$$
is syndetic, that is, $N_{P}(E)$ has a nontrivial intersection
with any long enough interval in $\Z$
(see \cite{BM1}).
The polynomial Szemer\'{e}di theorem
is an extension of Szemer\'{e}di's theorem on arithmetic progressions,
which corresponds to $p_{i}(n)=in$, $i=1,\ld,r$,
(see \cite{Sz} and \cite{Fo})
and of the S\`{a}rk\"{o}zy--Furstenberg theorem,
which corresponds to the case $r=1$
(see \cite{Sar}, \cite{Fo}, \cite{Fb}).

It is not hard to see that the condition of homogeneity, 
$p_{i}(0)=0$, $i=1,\ld,r$, in the polynomial Szemer\'{e}di theorem
is not superfluous.
(Consider, for example, $r=1$, $p(n)=2n+1$, $E=2\N$, 
or $r=1$, $p(n)=n^{2}+1$, $E=3\N$.)
On the other hand, it is also clear that this condition 
is not a necessary one.
For example, it is easy to see 
that it can be replaced by the condition $p_{i}(n_{0})=0$, $i=1,\ld,r$,
for some $n_{0}\in\Z$.
Actually, the latter condition still falls short of being necessary.
Let us say that a family of integral polynomials $P=\{p_{1},\ld,p_{r}\}$
has PSZ property if for every set $E\sle\Z$ with $d^{*}(E)>0$
the introduced above set $N_{P}(E)$ is nonempty,
and let us say that $P$ has SPSZ property 
if for every set $E\sle\Z$ with $d^{*}(E)>0$
the set $N_{P}(E)$ is syndetic.
Our goal in this paper
is to establish necessary and sufficient conditions
for a family of integral polynomials to have PSZ property.
When $r=1$, such a condition was obtained in \cite{KM}.
Namely, it was proved in \cite{KM}
that a family consisting of a single integral polynomial $p$ has PSZ property
if and only if $p$ is {\it intersective\/},
meaning that for any $k\in\N$
the intersection $\{p(n),\ n\in\Z\}\cap k\Z$ is nonempty.

As we will see, 
our condition for a family $P$ to have PSZ property
is a natural generalization of the Kamae and Mend\`{e}s-France condition.
We will say that polynomials $p_{1},\ld,p_{r}$ are {\it jointly intersective\/}
if for every $k\in\N$ there exists $n\in\Z$
such that $p_{i}(n)$ is divisible by $k$ for all $i=1,\ld,r$.
Here is now the formulation of our main result.
\begin{theorem}
\label{Th}
Let $P=\{p_{1},\ld,p_{r}\}$ be a system of integral polynomials.
The following statements are equivalent:\\
{\rm(i)} $P$ has PSZ property;\\
{\rm(ii)} $P$ has SPSZ property;\\
{\rm(iii)} the polynomials $p_{1},\ld,p_{r}$ are jointly intersective.
\end{theorem}

\begin{remark}
One can easily show (see Section \ref{s-OneVar} below)
that several integral polynomials of one variable are jointly intersective
if and only if they are all divisible by a single intersective polynomial,
and thus it follows from Theorem~\ref{Th}
that a family $P$ of integral polynomials possesses the PSZ property
iff it is of the form $P=\{q_{1}p,q_{2}p,\ld,q_{r}p\}$
where $q_{1},\ld,q_{r}\in\Q[n]$ and $p$ is an intersective polynomial.
In particular,
for any intersective polynomial $p$ and any $r\in\N$
the family $P=\{p,2p,\ld,rp\}$ has PSZ property;
this result was recently obtained by Frantzikinakis (\cite{Fra}).
\end{remark}

Theorem~\ref{Th} tells us that the only obstacle
for a family of integral polynomials to possess PSZ property
is of {\it arithmetic\/} nature.
The following direct corollary of Theorem~\ref{Th}
gives a precise meaning to this observation:
\begin{theorem}
\label{Thp}
If\/ $p_{1},\ld,p_{r}$ are integral polynomials
such that any lattice $k\Z$ in $\Z$ contains a configuration of the form
$\{a,a+p_{1}(n),\ld,a+p_{r}(n)\}$ with $a,n\in\Z$,
then any set of positive upper Banach density in $\Z$
also contains such a configuration.
\end{theorem}

As a matter of fact,
we will obtain a ``multiparameter'' version of Theorem~\ref{Th},
that is, we will prove this theorem for polynomials of several variables.
(Passing from one to many variables does not make the proof longer,
but essentially strengthens the theorem.)
We say that a polynomial $p$ of $m\geq 1$ variables 
with rational coefficients is {\it integral\/}
if $p(\Z^{m})\sle\Z$.
We will interpret any integral polynomial $p$ of $m$ variables 
as a mapping $\Z^{m}\ra\Z$,
and say that $p$ is {\it an integral polynomial on $\Z^{m}$}.
A set $S$ in $\Z^{m}$ is said to be {\it syndetic\/}
if $S+K=\Z^{m}$ for some finite $K\sln\Z^{m}$;
the rest of definitions do not change,
and, starting from this moment,
we will assume that the polynomials $p_{1},\ld,p_{r}$ in Theorem~\ref{Th}
are integral polynomials on $\Z^{m}$.

Clearly, (ii) in Theorem~\ref{Th} implies (i);
it is also clear that (i) implies (iii):
if $p_{1},\ld,p_{r}$ are not jointly intersective
and $k\in\N$ is such that for no $n\in\Z^{m}$
the integers $p_{1}(n),\ld,p_{r}(n)$ are all divisible by $k$,
the lattice $k\Z$ does not contain configurations
of the form $\{a,a+p_{1}(n),\ld,a+p_{r}(n)\}$.
So, it is only the implication $\hbox{(iii)}\Longrightarrow\hbox{(ii)}$
which needs to be proven.
We will actually get a stronger result:
\begin{theorem}
\label{P-Sz}
Let $p_{1},\ld,p_{r}$ be jointly intersective integral polynomials on $\Z^{m}$
and let $E\sle\Z$, $d^{*}(E)>0$.
Then there exists $\eps>0$ such that the set 
$$
\Bigl\{n\in\Z^{m}: 
d^{*}\bigl(E\cap(E-p_{1}(n))\cap\ld\cap(E-p_{r}(n))\bigr)>\eps\Bigr\}
$$
is syndetic.
\end{theorem}

Like the proof of the polynomial Szemer\'{e}di theorem in \cite{psz},
our proof of Theorem~\ref{P-Sz} 
relies on Furstenberg's correspondence principle.
This principle, which plays instrumental role in \cite{Fo},
can be found in the following form in \cite{Be}:\\
{\it For any set $E\sle\Z$ with $d^{*}(E)>0$
there exists an invertible probability measure preserving system $(X,\B,\mu,T)$
and a set $A\in\B$ with $\mu(A)=d^{*}(E)$
such that for any $r\in\N$ and $n_{1},n_{2},\ld,n_{r}\in\Z$ one has
$$
d^{*}\bigl(E\cap(E-n_{1})\cap\ld\cap(E-n_{r})\bigr)
\geq\mu\bigl(A\cap T^{-n_{1}}A\cap\ld\cap T^{-n_{r}}A\bigr).
$$}

For a multiparameter sequence $(a_{n})_{n\in\Z^{m}}$ of real numbers we define 
$\Clim_{n}a_{n}=\lim_{N\ras\infty}\frac{1}{|\Phi_{N}|}\sum_{n\in\Phi_{N}}a_{n}$,
if this limit exists for every F{\o}lner sequence $(\Phi_{N})$ in $\Z^{m}$.
(Note that if this limit exists for all F{\o}lner sequences, 
then it does not depend on the choice of the sequence.)
In view of Furstenberg's correspondence principle,
Theorem~\ref{P-Sz} is a corollary of the following ergodic result.
\begin{theorem}
\label{ThD}
Let integral polynomials $p_{1},\ld,p_{r}$ on $\Z^{m}$ be jointly intersective.
Then for any invertible probability measure preserving system $(X,\B,\mu,T)$
and any set $A\in\B$ with $\mu(A)>0$,
\begin{equation}
\label{f-pl}
\Clim_{n}\mu\bigl(A\cap T^{-p_{1}(n)}A\cap\ld\cap T^{-p_{r}(n)}A\bigr)>0.
\end{equation}
\end{theorem}

\noindent
(We remark that the converse of this theorem is also true:
if the polynomials $p_{1},\ld,p_{r}$ are not jointly intersective,
one can construct a (finite) measure preserving system and a set $A$
such that the limit in (\ref{f-pl}) is equal to 0.
We also remark that having ``$\liminf$'' 
instead of ``$\lim$'' in formula (\ref{f-pl})
would be quite sufficient to prove Theorem~\ref{P-Sz};
but, anyway, it is known that the limit
$\Clim_{n}\mu\bigl(A\cap T^{-p_{1}(n)}A\cap\ld\cap T^{-p_{r}(n)}A\bigr)$ exists,
-- see \cite{cpm}.)

It is worth noticing that while, being ergodic in nature,
our proof of Theorem~\ref{ThD} is quite different 
from the ergodic proofs of polynomial Szemer\'{e}di theorem
in \cite{psz} and \cite{BM1},
and hence provides a new proof 
of the ``homogeneous'' polynomial Szemer\'{e}di theorem as well.
The reason that we had to resort to a completely different approach
lies with the fact that the main ingredients 
of the proofs in \cite{psz} and \cite{BM1},
namely the PET induction and combinatorial results
such as the polynomial van der Waerden theorem (in \cite{psz})
and the polynomial Hales--Jewett theorem (in \cite{BM1}),
do not work when the polynomials involved may have a non-zero constant term.
In particular, it is not clear how to obtain by purely combinatorial means
(or with the help of topological dynamics
but without using an invariant measure)
the following corollary of Theorem~\ref{P-Sz}.
\begin{theorem}
\label{ThvdW}
For any finite partition of\/ $\Z$, $\Z=\bigcup_{i=1}^{k}E_{i}$,
one of $E_{i}$ has the property that for any $r$, $m$, 
and any jointly intersective integral polynomials $p_{1},\ld,p_{r}$ on $\Z^{m}$
there exists $\eps>0$ such that the set 
$\bigl\{n\in\Z^{m}:
d^{*}\bigl(E_{i}\cap(E_{i}-p_{1}(n))\cap\ld\cap(E_{i}-p_{r}(n))\bigr)>\eps\bigr\}$
is syndetic.
\end{theorem}
\noindent
\begin{remarks}
1.
One can also show that,
given a partition $\Z=\bigcup_{i=1}^{k}E_{i}$ of $\Z$,
for any collection $p_{1},\ld,p_{r}$ of integral polynomials on $\Z$
one of $E_{i}$ contains many configurations of the form
$\{a,a+p_{1}(n),\ld,a+p_{r}(n)\}$ with $n\in E_{i}$;
see Theorem~\ref{ThnvdW} below.

\noindent
2. 
Note that if $p_{1},\ld,p_{r}$ are not jointly intersective
and $k\in\N$ is such that for no $n\in\Z^{m}$
the integers $p_{1}(n),\ld,p_{r}(n)$ are all divisible by $k$,
then no element of the partition $\Z=\bigcup_{i=0}^{k-1}(k\Z+i)$ of $\Z$
contains configurations of the form $\{a,a+p_{1}(n),\ld,a+p_{r}(n)\}$.
\end{remarks}

The proof of Theorem~\ref{ThD} is divided into several steps. 
The first one is a reduction to nilsystems via Host-Kra--Ziegler machinery. 
The second step is a differential geometry argument (Lemma~\ref{geom}) 
which allows us to reduce the recurrence problem
to properties of the closure of an orbit in a nilsystem (Proposition~\ref{P-DinY}). 
The last step is a description of polynomial orbits on tori (Section~\ref{S-Tor}) 
and on nilmanifolds (Section~\ref{S-Nil}).
In Section~\ref{S-Sz} we finish the proof of Theorem~\ref{P-Sz}
and obtain (the enhanced version of) Theorem~\ref{ThvdW}.
Section~\ref{S-Misc} is devoted to concluding remarks and conjectures.
%=====================================================SSSSSSSSS
\section{Polynomial Szemer\'{e}di theorem
and polynomial orbits in nilmanifolds}
\label{S-Orb}
%----------------------
A {\it nilsystem\/} is a measure preserving system
defined by a translation $g\Gam\mapsto ag\Gam$
on a compact nilmanifold $X=G/\Gam$
(where $G$ is a nilpotent Lie group, 
$\Gam$ is a discrete uniform subgroup of $G$, and $a\in G$)
equipped with the (normalized) Haar measure, which will be denoted by $\mu$.
{\it A pro-nilsystem\/} is the inverse limit of a sequence of nilsystems.
Let $p_{1},\ld,p_{r}$ be integral polynomials on $\Z^{m}$, $m\geq1$.
It was proved in \cite{cpm} (see also \cite{HKp})
that for any probability measure preserving system $(X,T,\mu)$,
a certain pro-nilsystem $(\tX,\tT,\tmu)$ 
is {\it a characteristic factor\/} of $(X,T,\mu)$
with respect to the system of polynomial actions
$\{T^{p_{1}(n)},\ld,T^{p_{r}(n)}\}$,
which means that $(\tX,\tT,\tmu)$ is a factor of $(X,T,\mu)$ such that
for any $f_{0},f_{1},\ld,f_{r}\in L^{\infty}(X)$ one has
\begin{multline*}
\Clim_{n}
\int_{X}f_{0}\cdot f_{1}\comp T^{p_{1}(n)}\cdot\ld\cdot f_{r}\comp T^{p_{r}(n)}\,d\mu
\\
=\Clim_{n}
\int_{X}E(f_{0}|\tX)\cdot E(f_{1}|\tX)\comp\tT^{p_{1}(n)}\cdot\ld
\cdot E(f_{r}|\tX)\comp\tT^{p_{r}(n)}\,d\tmu
\end{multline*}
(where $E(\cd|\tX)$ stands for the conditional expectation with respect to $\tX$).

The statement
\begin{equation}
\label{poslim}
\Clim_{n}
\mu\bigl(A\cap T^{-p_{1}(n)}A\cap\ld\cap T^{-p_{r}(n)}A\bigr)>0
\hbox{ for any measurable $A\sle X$ with $\mu(A)>0$}
\end{equation}
is clearly equivalent to the statement
\begin{eqnarray*}
\Clim_{n}
\int_{X}f\cdot f\comp T^{p_{1}(n)}\cdot\ld\cdot f\comp T^{p_{r}(n)}d\mu>0
\quad\hbox{for all $f\in L^{\infty}(X)$}\kern35mm
\\
\hbox{such that $f\geq 0$ and $\int_{X}f\,d\mu>0$}.
\end{eqnarray*}
Thus, in order to prove Theorem~\ref{ThD},
we have to check (\ref{poslim}) for pro-nilsystems only.
The following lemma, which appears in \cite{FK}, 
shows that it is enough to check the result 
in the case where $(X,T,\mu)$ is a nilsystem.
\begin{lemma} 
Let $r\in\N$.
Let $(X,\B,\mu,T)$ be a measure preserving dynamical system 
and $(\B_\alf)_{\alf\geq1}$ be an increasing sequence 
of $T$-invariant sub-$\sig$-algebras 
such that $\bigvee_{\alf\geq1}\B_\alf=\B$. 
Then, for any $B\in\B$, 
there exists $\alf\geq1$ and $B'\in\B_\alf$ 
such that $\mu(B')\geq\mu(B)/2$ and, for all $n_1,\ld,n_r\in\Z$, 
$$
\mu\left(B\cap T^{-n_1}B\cap\ld\cap T^{-n_r}B\right)
\geq{\txt\frac{1}{2}}\mu\left(B'\cap T^{-n_1}B'\cap\ld\cap T^{-n_r}B'\right).
$$
\end{lemma}
\begin{proof}
We assume that $\mu(B)>0$. 
The sequence of conditional probabilities 
$\left(\mu\left(B\mid\B_\alf\right)\right)_{\alf\geq1}$ 
converges in probability to the characteristic function $1_B$. 
Hence there exists $\alf$ such that the set 
$B':=\left\{\mu\left(B\mid\B_\alf\right)\geq1-\frac1{2(r+1)}\right\}$ 
has measure $\geq\frac{1}{2}\mu(B)$. 
For any $n\in\Z$, we have 
$T^{-n}B':=\left\{\mu\left(T^{-n}B\mid\B_\alf\right)\geq1-\frac1{2(r+1)}\right\}$.
Using the fact that 
$\mu\left(B_0\cap B_1\cap\ld\cap B_r\mid\B_\alf\right)\geq1-(r+1)\eta$ 
if $\mu\left(B_i\mid\B_\alf\right)\geq1-\eta$, $0\leq i\leq r$, 
we have
\begin{eqnarray*}
\mu\left(B\cap T^{-n_1}B\cap\ld\cap T^{-n_r}B\right)
&=&
\int_{X}\mu\left(B\cap T^{-n_1}B\cap\ld
\cap T^{-n_r}B\mid \B_\alf\right)\,d\mu\\
&\geq&
\int_{B'\cap T^{-n_1}B'\cap\ld\cap T^{-n_r}B'}\kern-30mm
\mu\left(B\cap T^{-n_1}B\cap\ld\cap T^{-n_r}B\mid \B_\alf\right)\,d\mu\\
&\geq&
{\txt\frac{1}{2}}\mu\left(B'\cap T^{-n_1}B'\cap\ld\cap T^{-n_r}B'\right).
\end{eqnarray*}
\end{proof}

Thus we may and, from now on, will assume that $(X,\mu,T)$ is a nilsystem.
%--------------------

{\it A subnilmanifold\/} of $X$ is a closed subset of $X$ of the form $D=Kx$, 
where $K$ is a closed subgroup of $G$ and $x\in X$. 
A subnilmanifold is a nilmanifold itself 
under the action of the nilpotent Lie group $K$,
and supports a unique probability Haar measure which we will denote by $\mu_{D}$.

It is known (see \cite{pen}, or \cite{Shah} for a much more general result)
that if $H$ is a subgroup of $G$ and $x\in X$,
then $D=\overline{Hx}$ is a subnilmanifold of $X$.

{\it A (multiparameter) polynomial sequence in $G$\/}
is a mapping $g\col\Z^{m}\ra G$ of the form
$g(n)=a_{1}^{p_{1}(n)}\ld a_{r}^{p_{r}(n)}$, $n\in\Z^{m}$,
where $a_{i}\in G$ and $p_{i}$ are integral polynomials on $\Z^{m}$.
It is proved in \cite{mpn}
that if $g$ is a polynomial sequence in $G$
and $D$ is a subnilmanifold of $X$,
then the closure $Y=\corb_g(D)$ 
of the orbit $\orb_{g}(D)=\bigcup_{n\in\Z}g(n)D$ of $D$ 
is either a subnilmanifold or a finite disjoint union of subnilmanifolds of $X$.
Moreover, the sequence $\{g(n)D\}_{n\in\Z}$ 
has an asymptotic distribution in $Y$: 
we have $\Clim_{n}g(n)\mu_{D}=\mu'_{Y}$, 
where $\mu'_{Y}$ is a convex combination of the Haar measures 
on the connected components of $Y$. 
In particular, if $Y$ is connected, 
then $Y$ is a subnilmanifold,
and $\mu'_Y=\mu_Y$ is the Haar measure on $Y$.
%------------------

Let $p_{1},\ld,p_{r}$ be integral polynomials on $\Z^{m}$;
consider the polynomial sequence
$g(n)=\svect{1_{G}\\a^{p_{1}(n)}\\\vd\\a^{p_{r}(n)}}$, $n\in\Z^{m}$,
in the group $G^{r+1}$.
Let $\Del_{X^{r+1}}$ be {\it the diagonal\/},
$\Del_{X^{r+1}}=\Bigl\{\bx=\svect{x\\\svd\\x}:x\in X\Bigr\}$
in the nilmanifold $X^{r+1}$,
and let $Y=\corb_{g}(\Del_{X^{r+1}})$.
Then for any continuous functions $f_{0},f_{1},\ld,f_{r}$ on $X$,
\begin{eqnarray*}
&&\Clim_{n}
\int_{X}f_{0}\cdot f_{1}\comp T^{p_{1}(n)}\cdot\ld
\cdot f_{r}\comp T^{p_{r}(n)}\,d\mu
\\
&=&\Clim_{n}
\int_{\Del_{X^{r+1}}}f_{0}\tens f_{1}\comp T^{p_{1}(n)}\tens\ld
\tens f_{r}\comp T^{p_{r}(n)}\,d\mu_{\Del_{X^{r+1}}}
\\ 
&=&\Clim_{n}\int_{\Del_{X^{r+1}}}
\bigl(f_{0}\tens f_{1}\tens\ld\tens f_{r}\bigr)(g(n)\overline x)\, 
d\mu_{\Del_{X^{r+1}}}(\overline x)
\\
&=&\Clim_{n}
\int_{g(n)\Del_{X^{r+1}}}f_{0}\tens f_{1}\tens\ld
\tens f_{r}\,d\mu_{g(n)\Del_{X^{r+1}}}.
\\
&=&\int_{Y}f_{0}\tens f_{1}\tens\ld\tens f_{r}\,d\mu'_{Y}.
\end{eqnarray*}
Since $C(X)$ is dense in $L^{r+1}(X,\mu)$ 
and all the marginals of $\mu'_{Y}$ are equal to $\mu$,
we obtain by the multilinearity of the above expressions that
$$
\Clim_{n}
\int_{X}f_{0}\cdot f_{1}\comp T^{p_{1}(n)}\cdot\ld
\cdot f_{r}\comp T^{p_{r}(n)}\,d\mu
=\int_{Y}f_{0}\tens f_{1}\tens\ld\tens f_{r}\,d\mu'_{Y}
$$
for any $f_{0},f_{1},\ld,f_{r}\in L^{\infty}(X)$.
In particular, for any measurable set $A\sle X$,
$$
\Clim_{n}
\mu\bigl(A\cap T^{-p_{1}(n)}A\cap\ld\cap T^{-p_{r}(n)}A\bigr)
=\mu'_{Y}(A^{r+1}\cap Y),
$$
and in order to prove Theorem~\ref{ThD}
we only need to show that $\mu'_{Y}(A^{r+1}\cap Y)>0$ whenever $\mu(A)>0$.
%------------------------

We claim that this is true as long as $Y\sge\Del_{X^{r+1}}$.
Indeed, let us assume that this inclusion holds,
and let $A$ be a set of positive measure in $X$.
Let $x\in X$ be a Lebesgue point of $A$,
and let $\bx=\svect{x\\\svd\\x}\in\Del_{X^{r+1}}$.
Using a system of Malcev coordinates in $G$ (see Section~\ref{S-Nil}),
we identify a connected open neighborhood $\Omega$ of $x$
with an open subset of $\R^{d}$, where $d=\dim X$.
Then, under this identification, 
$Y'=Y\cap\Omega^{r+1}$ is a smooth (polynomial) manifold in $\R^{d(r+1)}$,
and the restriction on $Y'$ of the measure $\mu'_{Y}$
is equivalent to the Lebesgue measure
(that is, the $s$-volume, where $s=\dim Y$) in $Y'$.
Let $S$ be the connected component of $Y'$
that contains $\Del_{\Omega^{r+1}}$.
Our claim now follows from of the following lemma.
%--------------
\begin{lemma}
\label{geom}
Let $\Omega$ be an open subset of\/ $\R^{d}$
and let $S$ be a connected $C^{1}$-manifold in $\Omega^{k}$
with $S\sge\Del_{\Omega^{k}}$.
Let $\sig$ be the Lebesgue measure on $S$.
Then for any subset $A$ of\/ $\Omega$ with positive Lebesgue measure
one has $\sig(A^{k}\cap S)>0$.
\end{lemma}
\begin{proof}
Let $x$ be a density point of $A$.
For $t>0$ let $Q_{t}$ be the cube in $\R^{d}$ of size $t$ centered at $x$,
and let $P_{t}=Q_{t}^{k}$
(which is the cube in $\R^{dk}$ of size $t$ centered at $\bx=\svect{x\\\svd\\x}$).
Let $\pi_{i}$, $i=1,\ld,k$,
be the projection from $\R^{dk}=(\R^{d})^{k}$ onto the $i$th factor.
Since $S$ contains $\Del_{\Omega^{k}}$,
for any $i$,
$\pi_{i}$ projects $S$ onto $\Omega$ 
and has full rank at all points of $S$.

Let $s=\dim S$.
Let $L$ be the tangent space to $S$ at the point $\bx=\svect{x\\\svd\\x}$
and let $\lam$ be the Lebesgue measure (the $s$-volume) on $L$.
Let $1\leq i\leq k$.
If $t$ is small enough 
(so that, in particular, $Q_{2t}\subseteq\Omega$),
we have
\begin{equation}
\label{f-ineq1}
\sig\left(\pi_{i}^{-1}(B)\cap S\cap P_{t}\right)
\leq 2\lam\left(\pi_{i}^{-1}(B)\cap L\cap P_{2t}\right)
\end{equation}
and
\begin{equation}
\label{f-ineq2}
\lam\left(\pi_{i}^{-1}(B)\cap L\cap P_{t}\right)
\leq 2\sig\left(\pi_{i}^{-1}(B)\cap S\cap P_{2t}\right)
\end{equation}
for any measurable set $B\sle\Omega$.
Let $\sig_{t}$ and $\lam_{t}$
be the normalized Lebesgue measures
on $S\cap P_{t}$ and on $L\cap P_{t}$ respectively.
Then for $t$ small enough we have from (\ref{f-ineq2}) that
$\sig\left(S\cap P_{t}\right)
\geq\frac{1}{2}\lam\left(L\cap P_{t/2}\right)
=2^{-2s-1}\lam\left(L\cap P_{2t}\right)$,
and thus from (\ref{f-ineq1}),
\begin{equation}
\label{f-ineq3}
\sig_{t}(\pi_{i}^{-1}(B)\cap S\cap P_{t})
\leq 2^{2s+2}\lam_{2t}(\pi_{i}^{-1}(B)\cap L\cap P_{2t})
\end{equation}
for any measurable $B\sle\Omega$.

For $t>0$, let $\nu_{t}$ 
be the normalized Lebesgue measure on the cube $Q_{t}\subset\R^{d}$.
Since $L$ is an affine space passing through the center of $Q_{t}$,
and since, for each $i$, $L$ projects by $\pi_{i}$ {\sl onto} $\R^{d}$,
we have $\pi_{i}(\lam_{t})\leq c_{i}\nu_{t}$ 
with a constant $c_{i}$ independent on $t$.
Let $c=\max\{c_{1},\ld,c_{k}\}$,
then $\lam_{t}(\pi_{i}^{-1}(B)\cap L)\leq c\nu_{t}(B)$
for any measurable set $B\sle\R^{d}$ and all $i$.

Now choose $t$ small enough so that (\ref{f-ineq3}) holds for all $i$
and that $\nu_{2t}\left(Q_{2t}\sm A\right)<1/(2^{2s+2}kc)$. 
Then
\begin{multline*}
\sig_{t}\left(A^{k}\cap P_{t}\cap S\right)
\geq 1-\sum_{i=1}^{k}\sig_{t}\left(\pi_{i}^{-1}(Q_{t}\sm A)\cap S\right)
\geq 1-\sum_{i=1}^{k}2^{2s+2}\lam_{2t}\left(\pi_{i}^{-1}(Q_{2t}\sm A)\cap L\right)
\\
\geq 1-\sum_{i=1}^{k}2^{2s+2}c\nu_{2t}\left(Q_{2t}\sm A\right)>0,
\end{multline*}
and so $\sig(A^{k}\cap S)>0$.
\end{proof}
%------------------

Hence, we are done if we prove that $\corb(\bx)\ni\bx$ for every $\bx\in\Del_{X^{r+1}}$. 
After considering the new nilmanifold $X^{r+1}$ and changing notation, 
Theorem~\ref{ThD} is now reduced to the following proposition.
\begin{proposition}
\label{P-DinY}
Let $X=G/\Gam$ be a nilmanifold
and let $g(n)=a_{1}^{p_{1}(n)}\ld a_{r}^{p_{r}(n)}$ 
be a polynomial sequence in $G$
such that the polynomials $p_{1},\ld,p_{r}$ are jointly intersective.
Then $\corb_{g}(x)\ni x$ for any $x\in X$.
\end{proposition}
We will prove Proposition~\ref{P-DinY} in Section~\ref{S-Nil}
(see Proposition~\ref{P-PDinY}).
%=====================================================SSSSSSSSS
\section{Intersective polynomials and polynomial orbits on tori}
\label{S-Tor}
%-----------------------

Given two integers $b$, $k$, 
we will write $b\dvd k$ if $k$ divides $b$.
We will use the term {\it lattice\/} 
for cosets of subgroups of finite index in $\Z^{m}$.
If $\Lam$ is a lattice of $\Z^{m}$,
then $\Lam$ is itself isomorphic to $\Z^{m}$,
and the notion of an integral polynomial on $\Lam$ is well defined.
(Clearly, integral polynomials on $\Lam$ 
are restrictions of polynomials on $\Z^{m}$ taking on integer values on $\Lam$.)
We will say that integral polynomials $p_{1},\ld,p_{r}$ on $\Lam$
are jointly intersective (on $\Lam$)
if for any $k\in\N$ there exists $n\in\Lam$ 
such that $p_{1}(n),\ld,p_{r}(n)\dvd k$.
\begin{lemma}
\label{P-subdiv}
If integral polynomials $p_{1},\ld,p_{r}$ on a lattice $\Lam$
are jointly intersective,
then for any sublattice $\Lam'$ of $\Lam$ there exists $l\in\Lam$
such that the polynomials $p_{1},\ld,p_{r}$ 
are jointly intersective on $\Lam'+l$.
\end{lemma}
\begin{proof} 
Let $L\sln\Lam$ be a finite set such that $\Lam'+L=\Lam$.
For any $k\in\N$ there exists $l_{k}\in L$
such that $p_{i}(n+l_{k})\dvd k$, $i=1,\ld,r$, for some $n\in\Lam'$.
Let $l$ be such that $l_{k!}=l$ for infinitely many $k$.
Then for any $k\in\N$ there exists $k_{0}>k$ such that $l_{k_{0}!}=l$,
and thus there exists $n\in\Lam'$ 
such that $p_{i}(n+l)\dvd k_{0}!\dvd k$, $i=1,\ld,r$.
\end{proof}
%----------------------
\begin{lemma}
\label{P-subjint}
Let integral polynomials $p_{1},\ld,p_{r}$ on a lattice $\Lam$
be jointly intersective.
For any $k\in\N$ there exists a lattice $\Lam'\sle\Lam$
such that $p_{1},\ld,p_{r}$ are jointly intersective on $\Lam'$
and $p_{1}(n),\ld,p_{r}(n)\dvd k$ for all $n\in\Lam'$.
\end{lemma}
\begin{proof}
Let $d\in\N$ be such that $dp_{1},\ld,dp_{r}$ have integer coefficients.
By Lemma~\ref{P-subdiv}, 
there exists $l\in\Lam$ such that $p_{1},\ld,p_{r}$
are jointly intersective on $\Lam'=kd\Lam+l$.
There exists $n_{0}\in\Lam$ such that $p_{i}(kdn_{0}+l)\dvd k$, $i=1,\ld,r$.
For any $n\in\Lam$ and every $i$ we have
$p_{i}(kdn+l)=p_{i}(kdn_{0}+l)+q_{i}(kd(n-n_{0}))$
where $q_{i}$ is an integral polynomial with coefficients in $\frac{1}{d}\Z$ 
and zero constant term.
Hence, $q_{i}(kd(n-n_{0}))\dvd k$, $i=1,\ld,r$,
and so $p_{i}(kdn+l)\dvd k$, $i=1,\ld,r$, for all $n$.
\end{proof}
%-----------------------

Let $\T$ be an (additive) torus.
{\it A polynomial sequence\/} in $\T$ 
is a (multiparameter) sequence of the form
$t(n)=\sum_{i=1}^{r}p_{i}(n)v_{i}$, $n\in\Z^{m}$,
where $p_{i}$ are integral polynomials on $\Z^{m}$
and $v_{i}\in\T$, $i=1,\ld,r$.
It is well known (see \cite{Weyl})
that if $t$ is a polynomial sequence in $\T$,
then the closure $S=\overline{\{t(n)\}}_{n\in\Lam}$ of $t$ is a connected component,
or a union of several connected components,
of a coset $u+N$ for some closed subgroup $N$ of $\T$ and an element $u\in\T$.
%It is well known (see \cite{Weyl}) 
%that if $t$ is a polynomial sequence in $\T$,
%the closure $S=\overline{\{t(n)\}}_{n\in\Lam}$
%of $\{t(n)\}_{n\in\Lam}$ 
%consists of one or several connected components
%of a coset of a closed subgroup of $\T$;
In particular, if $S$ is connected, it is a subtorus of $\T$.
After choosing coordinates in $\T$
we identify $\T$ with a standard torus $\R^{s}/\Z^{s}$, $s\in\N$.
Then any polynomial sequence
$t(n)=\sum_{i=1}^{r}p_{i}(n)v_{i}$ in $\T$
can be written in the form
\begin{equation}
\label{f-polt}
t(n)=\left[\svect{q_{0,1}(n)\\\svd\\q_{0,s}(n)}{\txt\frac{1}{k}}+
\sum_{i=1}^{l}\svect{q_{i,1}(n)\\\svd\\q_{i,s}(n)}\alf_{i}\right]\mod\Z^{s},
\end{equation}
where $1,\alf_{1},\ld,\alf_{l}\in\R$ are rationally independent,
$k\in\N$,
and the polynomials $q_{i,j}$
are linear combinations, with integer coefficients,
of the polynomials $p_{1},\ld,p_{r}$.
%------------------------

We first take care of the ``irrational'' part of $t$.
For any polynomial $q$ let $\hq$ denote the polynomial $q-q(0)$.
\begin{lemma}
\label{P-irrorb}
{\rm(i)}
Let $t(n)=\svect{q_{1}(n)\\\svd\\q_{s}(n)}\alf\mod\Z^{s}$
where $\alf\in\R$ is irrational 
and $q_{1},\ld,q_{s}$ are integral polynomials on a lattice $\Lam$.
Then $\overline{\{t(n)\}}_{n\in\Lam}$ is the (connected) subtorus
$\left[\svect{q_{1}(0)\\\svd\\q_{s}(0)}\alf+
\Rspan\left\{\svect{\hq_{1}(n)\\\svd\\\hq_{s}(n)},\ n\in\Lam\right\}\right]
\mod\Z^{s}$ of\/ $\T$.\\
{\rm(ii)}
Let $b_{i}(n)=\svect{q_{i,1}(n)\\\svd\\q_{i,s}(n)}\mod\Z^{s}$
and $t_{i}=b_{i}\alf_{i}$, $i=1,\ld,l$,
where $1,\alf_{1},\ld,\alf_{l}\in\R$ are rationally independent 
and $q_{i,j}$ are integral polynomials on a lattice $\Lam$.
Let $t=\sum_{i=1}^{l}t_{i}$;
then $\overline{\{t(n)\}}_{n\in\Lam}=\sum_{i=1}^{l}\overline{\{t_{i}(n)\}}_{n\in\Lam}$.
In particular, $\overline{\{t(n)\}}_{n\in\Lam}$ 
is a (connected) subtorus of\/ $\T$.
\end{lemma}
\begin{proof}
(i)
We may assume that $q_{j}(0)=0$, $j=1,\ld,s$.
Let $\tS=\Rspan\left\{\svect{q_{1}(n)\\\svd\\q_{s}(n)},\ n\in\Lam\right\}\sle\R^{s}$
and $S=\tS\mod\Z^{s}$;
since the vectors $\svect{q_{1}(n)\\\svd\\q_{s}(n)}$ are rational, 
$S$ is closed in $\T$.
Hence $S$ is a subtorus and we have $\overline{\{t(n)\}}_{n\in\Lam}\sle S$.
On the other hand,
consider an additive character $\chi$ on $\T$,
$\chi\svect{v_{1}\\\svd\\v_{s}}=c_{1}v_{1}+\ld+c_{s}v_{s}\mod1$ 
with $c_{1},\ld,c_{s}\in\Z$;
if $\chi(t(n))=0$ for all $n\in\Lam$,
then $(c_{1}q_{1}(n)+\ld+c_{s}q_{s}(n))\alf\in\Z$ for all $n\in\Lam$,
so $c_{1}q_{1}(n)+\ld+c_{s}q_{s}(n)=0$ for all $n\in\Lam$,
so $\chi\rest{S}=0$.
Hence, the sequence $\{t(n)\}_{n\in\Lam}$ is not contained
in any proper closed subgroup of $S$,
and thus, is dense in $S$.
%------------
\par\noindent
(ii)
Again, we may assume that $q_{i,j}(0)=0$ for all $i,j$.
By (i), $\overline{\{t_{i}(n)\}}_{n\in\Lam}$, $i=1,\ld,l$,
are connected subgroups of $\T$,
and such is $N=\sum_{i=1}^{l}\overline{\{t_{i}(n)\}}_{n\in\Lam}$.
Let $S=\sum_{i=1}^{l}\overline{\{t_{i}(n)\}}_{n\in\Lam}$;
clearly, $S\sle N$.
We have $S\ni 0_{\T}$,
thus $S$ a union of connected components of a closed subgroup of $N$.

Let $\chi$ be a character on $\T$,
$\chi\svect{v_{1}\\\svd\\v_{s}}=c_{1}v_{1}+\ld+c_{s}v_{s}\mod1$ 
with $c_{1},\ld,c_{s}\in\Z$,
and let $\phi$ be the corresponding linear function on $\R^{s}$,
$\phi\svect{v_{1}\\\svd\\v_{s}}=c_{1}v_{1}+\ld+c_{s}v_{s}$. 
Then $\chi(t(n))=0$, $n\in\Lam$,
iff $\sum_{i=1}^{l}\phi(b_{i}(n))\alf_{i}=0\mod 1$, $n\in\Lam$,
which, because of the independence of $\alf_{1},\ld,\alf_{l}$ and $1$,
is equivalent to $\phi(b_{i}(n))=0$ and so, $\chi(t_{i}(n))=0$, $n\in\Lam$, 
for all $i=1,\ld,l$.
Hence, any character vanishing on $S$ also vanishes on $N$,
and so, $S$ is not contained in any proper closed subgroup of $N$.
Thus, $S=N$.
\end{proof}
%-----------------------
\begin{lemma}
\label{P-torzer}
Let $t(n)=\svect{q_{1}(n)\\\svd\\q_{s}(n)}\alf\mod\Z^{s}$
where $\alf\in\R$ is irrational 
and $q_{1},\ld,q_{s}$ are integral polynomials on a lattice $\Lam$.
Then $\overline{\{t(n)\}}_{n\in\Lam}\ni 0_{\T}$
iff no linear combination of $q_{1},\ld,q_{s}$ 
is a nonzero constant.
\end{lemma}
\begin{proof}
By Lemma~\ref{P-irrorb}(i),
$\overline{\{t(n)\}}_{n\in\Lam}\ni 0_{\T}$ iff
$\svect{q_{1}(0)\\\svd\\q_{s}(0)}\in
\Rspan\left\{\svect{\hq_{1}(n)\\\svd\\\hq_{s}(n)},\ n\in\Lam\right\}$.
This is so iff any linear function on $\R^{s}$ vanishing on 
$\Rspan\left\{\svect{\hq_{1}(n)\\\svd\\\hq_{s}(n)},\ n\in\Lam\right\}$
vanishes at $\svect{q_{1}(0)\\\svd\\q_{s}(0)}$ as well.
This is equivalent to saying that 
if $\sum_{i=1}^{s}c_{i}\hq_{i}=0$, with $c_{1},\ld,c_{s}\in\R$,
then also $\sum_{i=1}^{s}c_{i}q_{i}=0$.
\end{proof}
%-------------------
\begin{corollary}
\label{P-intzer}
Let $t(n)=\svect{q_{1}(n)\\\svd\\q_{s}(n)}\alf\mod\Z^{s}$
where $\alf\in\R$ is irrational 
and $q_{1},\ld,q_{s}$ are jointly intersective integral polynomials 
on a lattice $\Lam$.
Then $\overline{\{t(n)\}}_{n\in\Lam}\ni 0_{\T}$.
\end{corollary}
\proof{}
If there exist $c_{1},\ld,c_{s}\in\R$ and a nonzero $c\in\R$
such that $\sum_{i=1}^{s}c_{i}q_{i}=c$,
then, since $q_{i}$ have rational coefficients,
there exist $c_{1},\ld,c_{s}\in\Z$ and a nonzero $c\in\Z$
such that $\sum_{i=1}^{s}c_{i}q_{i}=c$.
But this is impossible if $q_{i}$ are jointly intersective.
\endproof
%-------------------

Let now $t$ be a polynomial sequence in $\T$,
$t(n)=p_{1}(n)v_{1}+\ld+p_{r}(n)v_{r}$, $v_{i}\in\T$,
where $p_{1},\ld,p_{r}$ are jointly intersective polynomials on $\Lam$.
\begin{proposition}
\label{P-divzer}
There exists a sublattice $\Lam'$ of $\Lam$
such that $p_{1},\ld,p_{r}$ are jointly intersective on $\Lam'$,
$S=\overline{\{t(n)\}}_{n\in\Lam'}$ is a connected subtorus of\/ $\T$,
and $0_{\T}\in S$.
\end{proposition}
\begin{proof}
We represent $t$ in the form (\ref{f-polt}),
where all polynomials $q_{i,j}$ are linear combinations of polynomials $p_{i}$
and so, are jointly intersective.
If a nontrivial ``rational'' term $\svect{q_{0,1}\\\svd\\q_{0,s}}\frac{1}{k}$
is present,
by Lemma~\ref{P-subjint} there exists a sublattice $\Lam'\sln\Lam$
such that the polynomials $q_{0,1},\ld,q_{0,r}$ 
are jointly intersective on $\Lam'$
and $q_{0,j}(n)\dvd k$ for all $n\in\Lam'$ and $j=1,\ld,s$.
Then $\svect{q_{0,1}(n)\\\svd\\q_{0,s}(n)}\frac{1}{k}=0\mod\Z^{s}$
for all $n\in\Lam'$, and we may ignore this term.
By Corollary~\ref{P-intzer},
for each $i=1,\ld,l$
and $t_{i}(n)=\svect{q_{i,1}(n)\\\svd\\q_{i,s}(n)}\alf_{i}\mod\Z^{s}$,
$S_{i}=\overline{\{t_{i}(n)\}}_{n\in\Lam'}$ is a (connected) subtorus of $\T$
with $0_{\T}\in S_{i}$,
and by Lemma~\ref{P-irrorb}(ii),
$S=\overline{\{t(n)\}}_{n\in\Lam'}=\sum_{i=1}^{l}S_{i}$.
Thus, $S$ is a (connected) subtorus of $\T$ with $0_{\T}\in S$.
\end{proof}
%=====================================================SSSSSSSSS
\section{Intersective polynomials and polynomial orbits on nilmanifolds}
\label{S-Nil}
%--------------------
Let $P$ be a ring of integral polynomials on a lattice $\Lam$.
We will say that a mapping $g$ from $\Lam$
to a nilpotent group $G$ is {\it a $P$-polynomial sequence\/}
if $g$ has the form $g(n)=a_{1}^{p_{1}(n)}\ld a_{r}^{p_{r}(n)}$
with $r\in\N$, $a_{i}\in G$ and $p_{i}\in P$, $i=1,\ld,r$. 
The following facts are obvious and will be used repeatedly in the sequel.\\
{\rm(i)}
if $g_{1}$, $g_{2}$ are $P$-polynomial sequences in $G$,
then the sequence $g_{1}(n)g_{2}(n)$ is $P$-polynomial;\\
{\rm(ii)}
if $\eta\col G\ra G'$ is a homomorphism to a nilpotent group $G'$
and $g$ is a $P$-polynomial sequence in $G$,
then $\eta(g)$ is a $P$-polynomial sequence in $G'$;\\
{\rm(iii)}
if $\eta\col G\ra G'$ is a homomorphism onto a nilpotent group $G'$
and $g'$ is a $P$-polynomial sequence in $G'$,
then there exists a $P$-polynomial sequence $g$ in $G$
such that $\eta(g)=g'$.

%----------------
\begin{proposition}
\label{P-polP} 
Let $G$ be a connected nilpotent Lie group 
and $H$ be a connected closed subgroup of $G$. 
If $g$ is a $P$-polynomial sequence in $G$ 
such that $g(n)\in H$ for all $n\in\Lam$,
then $g$ is a $P$-polynomial sequence in $H$.
\end{proposition}
%---------------
\begin{remark}
Actually, the assertion of Proposition~\ref{P-polP} holds 
for any (not necessarily topological) nilpotent group and any its subgroup
(see \cite{pon}).
\end{remark}
%---------------
\begin{proof}
Replacing $G$ by its universal cover
we may assume that $G$ is simply-connected.
We then may choose a Malcev basis in $G$,
that is, elements $e_{1},\ld,e_{k}\in G$
such that every element of $G$ is uniquely representable in the form
$\prod_{j=1}^{k}e_{j}^{y_{j}}$ with $y_{1},\ld,y_{k}\in\R$.
(See \cite{Malcev}.
Elements $e_{i}$ can be chosen to be of the form $e_i=\exp(\epsilon_i)$ 
where $(\epsilon_1,\ld,\epsilon_k)$ is a linear base of the Lie algebra of $G$.)
Moreover, by an elementary linear algebra argument,
the basis can be chosen compatible with $H$,
so that for some $j_{1},\ld,j_{l}\in\{1,\ld,k\}$,
the elements $e_{j_{1}},\ld,e_{j_{l}}$ form a basis in $H$,
and thus $\prod_{j=1}^{k}e_{j}^{y_{j}}\in H$ 
iff $y_{j}=0$ for all $j\not\in\{j_{1},\ld,j_{l}\}$.

From the Campbell-Hausdorff formula
we know that multiplication in $G$ is polynomial in the Malcev basis,
that is,
$\bigl(\prod_{j=1}^{k}e_{j}^{y_{j}}\bigr)\cdot
\bigl(\prod_{j=1}^{k}e_{j}^{z_{j}}\bigr)
=\prod_{j=1}^{k}e_{j}^{Q_{j}(y_{1},\ld,y_{k},z_{1},\ld,z_{k})}$
and $\bigl(\prod_{j=1}^{k}e_{j}^{y_{j}}\bigr)^{n}
=\prod_{j=1}^{k}e_{j}^{R_{j}(y_{1},\ld,y_{k},n)}$
where $Q_{j}$ and $R_{j}$ are polynomials vanishing at 0.
Thus, any polynomial sequence $g(n)=a_{1}^{p_{1}(n)}\ld a_{r}^{p_{r}(n)}$ in $G$
can be uniquely written as
$g(n)=\prod_{j=1}^{k}e_{j}^{F_{j}(p_{1}(n),\ld,p_{r}(n))}$
where $F_{j}$ are polynomials vanishing at 0.
If $g$ takes values only in $H$,
$F_{j}(p_{1}(n),\ld,p_{r}(n))=0$ for all $j\not\in\{j_{1},\ld,j_{l}\}$,
and $g(n)=\prod_{j\in\{j_{1},\ld,j_{l}\}}e_{j}^{F_{j}(p_{1}(n),\ld,p_{r}(n))}$
is a polynomial sequence in $H$.
The last formula can be rewritten as
$g(n)=\prod_{j\in\{j_{1},\ld,j_{l}\}}\prod_{i=1}^{k_{j}}
(e_{j}^{\alf_{j,i}})^{F_{j,i}(p_{1}(n),\ld,p_{r}(n))}$
where $\alf_{j,i}\in\R$ and $F_{j,i}$ are nonconstant monomials.
Now, if all $p_{i}$ are in $P$,
the polynomials $q_{j,i}(n)=F_{j,i}(p_{1}(n),\ld,p_{r}(n))$ are also in $P$,
and so, $g$ is a $P$-polynomial sequence in $H$.
\end{proof}
%--------------------

We will also need the following fact:
\begin{proposition}
\label{P-maxtor}
{\rm(\cite{pen})}
Let $G$ be a connected nilpotent Lie group,
let $X=G/\Gam$ be a nilmanifold,
let $\pi$ be the canonical projection $G\ra X$,
let $\T$ be the torus $\lfa{X}{[G,G]}$,
and let $\xi\col X\ra\T$ be the projection.
If a polynomial sequence $g$ in $G$
is such that $\xi(\pi(g(n)))$ is dense in $\T$,
then $\pi(g(n))$ is dense in $X$.
\end{proposition}
%--------------------

Now let $G$ be a nilpotent group,
$\Gam$ a closed uniform subgroup of $G$,
and $X=G/\Gam$.
Let $\pi$ be the projection $G\ra X$,
and $1_{X}=\pi(1_{G})\in X$.
Let $a_{1},\ld,a_{r}\in G$,
let $p_{1},\ld,p_{r}$ be jointly intersective polynomials
on a lattice $\Lam$,
and let $P$ be the ring generated by the polynomials $p_{1},\ld,p_{r}$.
Proposition~\ref{P-DinY} is a consequence of the following proposition,
applied to $g(n)=a_{1}^{p_{1}(n)}\ld a_{r}^{p_{r}(n)}$:
\begin{proposition}
\label{P-PDinY}
If $g$ is a $P$-polynomial sequence in $G$ and $x\in X$, 
then $\overline{\{g(n)x\}}_{n\in\Lam}\ni x$.
\end{proposition}
\begin{proof}
It is enough to prove that, for any $P$-polynomial sequence $g$, 
we have $\overline{\{\pi(g(n))\}}_{n\in\Lam}\ni 1_{X}$. 
Indeed, if $x=g_0\Gam\in X$ then $g_0^{-1}gg_0 $ 
is a $P$-polynomial sequence and $\overline{\{g(n)x\}}_{n\in\Lam}\ni x$ 
iff $\overline{\{\pi(g_0^{-1}g(n)g_0)\}}_{n\in\Lam}\ni 1_{X}$.

If $X$ is not connected,
let $\hG$ be a subgroup of finite index $k$ in $G$
such that $\Xc=\pi(\hG)$ is the identity component of $X$.
By Lemma~\ref{P-subjint}, there exists a sublattice $\Lam'$ of $\Lam$
such that the polynomials $p_{1},\ld,p_{r}$ 
are jointly intersective on $\Lam'$
and for any $n\in\Lam'$, $p_{1}(n),\ld,p_{r}(n)\dvd k$. 
The sequence $g\rest{\Lam'}$ takes values in $\hG$,
and after replacing $\Lam$ by $\Lam'$, $G$ by $\hG$, and $X$ by $\Xc$
we may assume that $X$ is connected.

Let $\Gc$ be the identity component of $G$
and let $\theta$ be the canonical homomorphism $G\ra G/\Gc$.
Since $X$ is connected, $\theta(\Gam)=G/\Gc$,
and thus there exists a $P$-polynomial sequence $\del$ in $\Gam$
such that $\theta(\del)=\theta(g)$.
The sequence $g'(n)=g(n)\del(n)^{-1}$ 
takes values in $\Gc$ and satisfies $\pi(g')=\pi(g)$, $n\in\Lam$.
%$$
%X=G/\Gam\hookleftarrow\Xc=\Gc/(\Gam\cap\Gc).
%$$
By Proposition~\ref{P-polP}, $g'(n)$ is a $P$-polynomial sequence in $\Gc$.
After replacing $g$ by $g'$ and $G$ by $\Gc$
we may assume that $G$ is connected.

Let $V=G/[G,G]=[G,G]\backslash G$ 
with $\eta\col G\ra V$ being the canonical projection.
$V$ is a connected commutative Lie group.
Let $\T$ be the torus $V/\eta(\Gam)=\lfa{X}{[G,G]}$
with $\tau\col V\ra\T$ being the projection;
we will use multiplicative notation for $V$ and $\T$.
Let $t(n)=g(n)1_{\T}$, $n\in\Lam$;
in other words, $t=\tau(\eta(g))$ is the projection of $g$ on $\T$.
$$
\begin{matrix} 
&\Gc&\ni&g(n)&&H\\
&\downarrow\scr\eta&&\downarrow&&\downarrow\\
&V=\Gc/[\Gc,\Gc]&\ni&\eta(g(n))&&L\\
&\downarrow\scr\tau&&\downarrow&&\downarrow\\
M=&\Gc/\left([\Gc,\Gc](\Gam\cap\Gc)\right)&\ni&t(n)=\tau(\eta(g(n))&\in&S
\end{matrix}
$$
If $t$ is dense in $\T$, then by Proposition~\ref{P-maxtor}, 
$g$ is dense in $X$ and we are done.
Assume that $t$ is not dense in $\T$.
We know that $t$ is a $P$-polynomial sequence in $\T$.
By Proposition~\ref{P-divzer},
after replacing $\Lam$ by a suitable sublattice,
the polynomials $p_{1},\ld,p_{r}$ remain jointly intersective
and $S=\overline{\{t(n)\}}_{n\in\Lam}$ is a connected proper subtorus of $\T$
with $1_{\T}\in S$.

Note that $\tau^{-1}(S)$ is a proper subgroup of $V$. 
Let $L\sle V$ be the identity component of $\tau^{-1}(S)$. 
We have $\tau(L)=S$.
Let $u$ be a $P$-polynomial sequence in $L$ such that $\tau(u)=t$.
Then $\tau(\eta(g))=\tau(u)$,
thus $u(n)^{-1}\eta(g(n))\in\eta(\Gam)$, $n\in\Lam$.
The sequence $\lam(n)=u(n)^{-1}\eta(g(n))$, $n\in\Lam$,
is $P$-polynomial in $\eta(\Gam)$;
let $\gam$ be a $P$-polynomial sequence in $\Gam$
such that $\eta(\gam)=\lam$.
Put $h(n)=g(n)\gam(n)^{-1}$, $n\in\Lam$;
then $\pi(h)=\pi(g)$ and $\eta(h)=u$.

Let $H=\eta^{-1}(L)$;
then $H$ is a proper closed connected subgroup of $G$,
and $Y=\pi(H)$ is a subnilmanifold of $X$ 
that contains the sequence $\pi(h)=\pi(g)$.
The sequence $h$ takes values in $H$,
thus by Proposition~\ref{P-polP}, $h$ is a $P$-polynomial sequence in $H$.
By induction on the dimension of $H$,
%$\overline{\{\pi(h(n))\}}_{n\in\Lam}\ni 1_H\Gam\cap H=1_G\Gam=1_{X}$.
$\overline{\{\pi(h(n))\}}_{n\in\Lam}\ni 1_{Y}=1_{X}$.
\end{proof}
%====================================================
\section{Polynomial Szemer\'{e}di and van der Waerden theorems}
\label{S-Sz}
%----------------
\begin{proof}[Proof of Theorem~\ref{P-Sz}]
By Furstenberg's correspondence principle,
there exists a probability measure preserving system $(X,\B,\mu,T)$
and a set $A\in\B$ with $\mu(A)=d^{*}(E)$
such that for any $n_{1},\ld,n_{l}\in\Z$ one has
$d^{*}\bigl(E\cap(E-n_{1})\cap\ld\cap(E-n_{l})\bigr)
\geq\mu\bigl(A\cap T^{-n_{1}}A\cap\ld\cap T^{-n_{l}}A\bigr)$.
Let $c_{n}=\mu\bigl(A\cap T^{-p_{1}(n)}A\cap\ld\cap T^{-p_{r}(n)}A\bigr)$,
$n\in\Z^{m}$.
By Theorem~\ref{ThD},
$\lim_{N-M\ras\infty}\frac{1}{(N-M)^{m}}\sum_{n\in[M,N-1]^{m}}c_{n}=C>0$,
and thus $d_{*}\bigl(\{n\in\Z^{m}:c_{n}>C/2\}\bigr)>0$,
where $d_{*}(F)=\liminf_{N-M\ras\infty}\frac{|F\cap[M,N-1]^{m}|}{(N-M)^{m}}$).
This means that the set $\{n\in\Z^{m}:c_{n}>C/2\}$ is syndetic.
\end{proof}
%----------------
The polynomial van der Waerden theorem for jointly intersective polynomials, 
Theorem~\ref{ThvdW}, 
is an immediate corollary of Theorem~\ref{P-Sz}. 
However, using a ``uniformity'' in Theorem~\ref{ThD}
(and following an idea which was utilized in \cite{BM1}),
we can get a stronger version of Theorem~\ref{ThvdW}.
We start with the following strengthening of the preceding theorem.
\begin{proposition}
\label{P-cvdW}
Let $p_{1},\ld,p_{r}$ be jointly intersective integral polynomials on $\Z^{m}$
and let sets $E_{1},\ld,E_{s}\sle\Z$ 
be such that $d^{*}(E_{i})>0$ for all $i=1,\ld,s$.
Then there exists $\eps>0$ such that the set 
\begin{equation}
\label{f-S}
S=\bigcap_{i=1}^{s}\Bigl\{n\in\Z^{m}: 
d^{*}(E_{i}\cap(E_{i}-p_{1}(n))\cap\ld\cap(E_{i}-p_{r}(n))\bigr)>\eps\Bigr\}
\end{equation}
is syndetic.
\end{proposition}
\begin{proof}
(Cf.\nasp\ the proof of Theorem~0.4 in \cite{BM1}.)
Using Furstenberg's correspondence principle,
for each $i=1,\ld,s$
find a probability measure preserving system $(X_{i},\B_{i},\mu_{i},T_{i})$
and a set $A_{i}\in\B_{i}$ with $\mu(A_{i})=d^{*}(E_{i})$
such that for any $n_{1},\ld,n_{l}\in\Z$ one has
$d^{*}\bigl(E_{i}\cap(E_{i}-n_{1})\cap\ld\cap(E_{i}-n_{l})\bigr)
\geq\mu_{i}\bigl(A_{i}\cap T_{i}^{-n_{1}}A_{i}\cap\ld
\cap T_{i}^{-n_{l}}A_{i}\bigr)$.
Put $X=X_{1}\times\ld\times X_{s}$, $T=T_{1}\times\ld\times T_{s}$,
and $A=A_{1}\times\ld\times A_{s}$.
By Theorem~\ref{ThD},
there exists $\eps>0$ such that the set
\begin{multline*}
\Bigl\{n\in\Z^{m}:
\mu\bigl(A\cap T^{-p_{1}(n)}A\cap\ld\cap T^{-p_{r}(n)}A\bigr)>\eps\Bigr\}
\\
=\Bigl\{n\in\Z^{m}:\prod_{i=1}^{s}
\mu_{i}\bigl(A_{i}\cap T^{-p_{1}(n)}A_{i}\cap\ld\cap T^{-p_{r}(n)}A_{i}\bigr)>\eps\Bigr\}
\end{multline*}
is syndetic,
and this is a subset of
$$
\bigcap_{i=1}^{s}\Bigl\{n\in\Z^{m}:
\mu_{i}\bigl(A_{i}\cap T^{-p_{1}(n)}A_{i}\cap\ld\cap T^{-p_{r}(n)}A_{i}\bigr)>\eps\Bigr\}.
$$
\end{proof}
%---------------------
We now confine ourselves to the one-parameter situation.
A subset $E$ of $\Z$ is said to be {\it piecewise syndetic\/}
if there exists a sequence of intervals $J_{1},J_{2},\ld$
with $|J_{j}|\ra\infty$ and a syndetic set $E'\sle\Z$
such that $E=E'\cap\bigcup_{j=1}^{\infty}J_{j}$.
It is not hard to see that
if a syndetic set is partitioned into finitely many subsets,
then one of these subsets is piecewise syndetic.
\begin{theorem}
\label{ThnvdW}
Let $p_{1},\ld,p_{r}$ be jointly intersective integral polynomials.
For any finite partition of\/ $\Z$, $\Z=\bigcup_{i=1}^{k}E_{i}$,
one of $E_{i}$ has the property that, for some $\eps>0$, the set
$$
\Bigl\{n\in E_{i}: 
d^{*}\bigl(E_{i}\cap(E_{i}-p_{1}(n))\cap\ld\cap(E_{i}-p_{r}(n))\bigr)>\eps
\Bigr\}
$$
is piecewise syndetic.
\end{theorem}
\begin{remark}
As it was already mentioned above,
the fact that for some $E_{i}$
(and indeed for any $E_{i}$ that has positive upper density)
and some $\eps>0$ the set
$$
\Bigl\{n\in\Z: 
d^{*}\bigl(E_{i}\cap(E_{i}-p_{1}(n))\cap\ld\cap(E_{i}-p_{r}(n))\bigr)>\eps
\Bigr\}
$$
is syndetic is a direct corollary of Theorem~\ref{P-Sz}.
The delicate point in Theorem~\ref{ThnvdW} 
is that the set of $n$ satisfying the assertion of the theorem
is a (large) subset of $E_{i}$.
\end{remark}
\begin{proof}
Re-index $E_{1},\ld,E_{k}$ so that $d^{*}(E_{i})>0$ for $i=1,\ld,s$
and $d^{*}(E_{i})=0$ for $i=s+1,\ld,k$.
Choose $\eps$ as in Proposition~\ref{P-cvdW},
and let $S$ be the syndetic set defined by (\ref{f-S}).
Since the set $\Z\sm\bigcup_{i=1}^{s}E_{i}$ has zero upper Banach density,
the set $S\cap\bigcup_{i=1}^{s}E_{i}$ is also syndetic,
and thus $S\cap E_{i}$ is piecewise syndetic for some $i\in\{1,\ld,s\}$.
\end{proof}
%====================================================
\section{Concluding remarks}
\label{S-Misc}
%---------------
\subsection{Intersective and jointly intersective polynomials}
\label{s-OneVar}

While every integral polynomial with an integer root is clearly intersective,
there are also examples of intersective polynomials without rational roots.
For example, one can show that if $a_{1}$, $a_{2}$
are distinct prime integers such that $a_{1}\equiv a_{2}\equiv 1\pmod4$
and $a_{1}$ is a square in $\Z/(a_{2}\Z)$,
then the polynomial $p(n)=(n^{2}-a_{1})(n^{2}-a_{2})(n^{2}-a_{1}a_{2})$
is intersective.
(Such is, for example, the polynomial $p(n)=(n^{2}-5)(n^{2}-41)(n^{2}-205)$.)
There are also similar examples of intersective polynomials of degree 5
(for instance, $p(n)=(n^{3}-19)(n^{2}+n+1)$),
and one can show (see \cite{BB})
that there exist no intersective polynomials in one variable of degree less than 5
without rational roots.
A curious example of an intersective polynomial of several variables
with no rational roots
is $p(n_{1},\ld,n_{4})=n_{1}^{2}+\ld+n_{4}^{2}+b$,
where $b$ is an arbitrary positive integer;
this polynomial has the property that all its shifts $p+c$, $c\in\Z$,
are also intersective.
(No intersective polynomials in one variable, 
except the polynomials $\pm n+b$, $b\in\Z$,
have this property.
Indeed, if an integral polynomial $p(n)$ is not of the form $\pm n+b$,
then there exists $n_{0}\in\Z$ such that $k=|p(n_{0}+1)-p(n_{0})|\neq 1$.
Then $p$ is not one-to-one in $\Z/(k\Z)$,
so is not onto,
and thus there exists $d\in\Z$ such that $p(n)-d\neq 0\mod k$ for any $n\in\Z$.)

\medbreak
Systems of jointly intersective polynomials in one variable
can be easily described:
\begin{proposition}
\label{P-comdiv}
Integral polynomials $p_{1},\ld,p_{r}$ of one variable are jointly intersective
iff they all are multiples of an intersective polynomial $p$.
\end{proposition}
\noindent
(We say that a polynomial $q$ is a multiple of a polynomial $p$
if $q$ is divisible by $p$ in the ring $\Q[n]$.)
\begin{proof}
Clearly, if $p\in\Q[n]$ is an intersective polynomial and $p_{1},\ld,p_{r}\dvd p$
then $p_{1},\ld,p_{r}$ are jointly intersective.

Let $p_{1},\ld,p_{r}\in\Q[n]$ be jointly intersective.
Let $p\in\Z[n]$ be the greatest common divisor of $p_{1},\ld,p_{r}$ in $\Q[n]$.
Then there exist $h_{1},\ld,h_{r}\in\Q[n]$
such that $\sum_{i=1}^{r}h_{i}p_{i}=p$.
Multiplying both parts by an integer $d$ if necessary,
we may assume that $h_{1},\ld,h_{r}$ have integer coefficients,
and that $\sum_{i=1}^{r}h_{i}p_{i}=dp$.
It is then clear that if $p_{1},\ld,p_{r}$ are jointly intersective, 
then $dp$ is intersective, and thus $p$ is intersective.
\end{proof}

The natural conjecture that integral polynomials are jointly intersective
if any linear combination of these polynomials is intersective, 
fails to be true.
For example, one can show that the polynomials 
$p_{1}(n)=n(n+1)(2n+1)$ and $p_{2}(n)=(n^{3}+n^{2}+2)(2n+1)$
satisfy the above condition,
but are not jointly intersective
(see Appendix in \cite{BeLe}).

Proposition~\ref{P-comdiv} is no longer true
for jointly intersective polynomials of several variables.
If polynomials $p_{1},\ld,p_{r}$ in $m$ variables
are jointly intersective,
then the whole ideal $I$ in $\Q[n_{1},\ld,n_{m}]$ generated by these polynomials
consists of jointly intersective polynomials.
In the case $m=1$, $I$ is principal,
from which Proposition~\ref{P-comdiv} follows.
If $m\geq 2$, $\Q[n_{1},\ld,n_{m}]$ is not a principal ideal domain,
and Proposition~\ref{P-comdiv} fails.
(Consider, for example, the pair of jointly intersective polynomials
$p_{i}(n_{1},n_{2})=n_{i}$, $i=1,2$.)
%-----------------
\subsection{Total ergodicity}

If one deals with totally ergodic dynamical systems
(this means that $T^{k}$ is ergodic for any nonzero integer $k$),
it is not hard to verify
(see Proposition~\ref{P-toterg} below)
that any integral polynomial is ``good'' for single recurrence.
This is no longer true for multiple recurrence,
as the simple example following Proposition~\ref{P-toterg} shows.

\begin{proposition}
\label{P-toterg}
Let $(X,\B,\mu,T)$ be a totally ergodic 
probability measure preserving dynamical system 
and let $p$ be an integral polynomial on $\Z^{m}$. 
Then, for any set $A\in\B$,
$\Clim_{n}\mu\bigl(A\cap T^{-p(n)}A\bigr)=\mu(A)^2$.
\end{proposition}
\begin{proof}
Total ergodicity of $T$ is equivalent 
to the lack of discrete rational spectrum 
for the unitary operator $f\mapsto f\comp T$ on $L^{2}(X)$.
For any $f\in L^{2}(X)$ 
and any F{\o}lner sequence $(\Phi_{N})_{N=1}^{\infty}$ in $\Z^{m}$, 
the convergence in $L^{2}$ of the sequence 
$\bigl(\frac{1}{|\Phi_{N}|}\sum_{n\in\Phi_{N}}f\comp T^{p(n)}\bigr)_{N=1}^{\infty}$ 
to the limit $\int f\,d\mu$ 
is then a consequence of basic spectral theory
and Weyl's equidistribution theorem.
(Cf.\nasp~\cite{Fb}, p.\nasp~70-71.)
\end{proof}

%---------------
\begin{example}
An example of a totally ergodic 
probability measure preserving dynamical system 
is the rotation of the one dimensional torus by an irrational number $\alf$. 
The simplest example of a non-intersective polynomial is $2n+1$. 
If we choose $A$ to be a sufficiently small interval on the torus, 
then, for any $n\neq 0$, 
we will have $A\cap T^{-n}A\cap T^{-(2n+1)}A=\emptyset$.
\end{example}

%---------------
It is natural to ask what is a necessary and sufficient condition
for a family $P=\{p_{1},\ld,p_{r}\}$ of integral polynomials 
to have ``the multiple recurrence property'' 
(namely, that for any $A\sle X$ with $\mu(A)>0$ 
one has $\mu(A\cap T^{-p_{1}(n)}A\cap\ld\cap T^{-p_{r}(n)}A)>0$ for a certain $n$)
in the framework of totally ergodic dynamical systems.
We conjecture that the condition 
that the ring generated by $p_{1},\ld,p_{r}$ does not contains nonzero constants
is a sufficient one.
However, this condition is far from being necessary;
for example, if the polynomials $p_{1},\ld,p_{r}$ are linearly independent,
it suffices that $\Zspan\{p_{1},\ld,p_{r}\}$ does not contain nonzero constants.
In order to find a necessary and sufficient condition 
for a family $P=\{p_{1},\ld,p_{r}\}$ of polynomials 
to have the multiple recurrence property under the assumption of total ergodicity
one has to take into consideration 
{\it the complexity\/} of the family $\{p_{1},\ld,p_{r}\}$
(see \cite{BLL1} and \cite{pod}).
Such a condition, however, would be too cumbersome
to be either of practical or aesthetic value.
%-----------------
\subsection{Multidimensional conjecture}

The {\it multidimensional\/} polynomial Szemer\'{e}di theorem
states that given a set $E$ of positive upper Banach density in $\Z^{k}$
and vector-valued polynomials $p_{1},\ld,p_{r}\col\Z^{m}\ra\Z^{k}$ 
with zero constant term,
the set 
$$
N_{P}(E)=\Bigl\{n\in\Z^{m}:\hbox{for some $a\in\Z^{k}$, 
$\{a,a+p_{1}(n),\ld,a+p_{r}(n)\}\sln E$}\Bigr\}
$$
is infinite, and, moreover, syndetic. 
(See \cite{psz} and \cite{BM2}.)
It is natural to try to generalize Theorem~\ref{Th} to this multidimensional situation.
Let us say that  a family $\{p_{1},\ld,p_{r}\}$ 
of polynomial mappings $\Z^{m}\ra\Z^{k}$ {\it has SPSZ property\/}
if for any set $E$ of positive upper Banach density in $\Z^{k}$
the set $N_{P}(E)$ is syndetic in $\Z^{m}$;
let us say that $p_{1},\ld,p_{r}$ are {\it jointly intersective\/}
if for any subgroup $\Lam$ of finite index in $\Z^{k}$
there exists $n\in\Z^{m}$ such that $p_{1}(n),\ld,p_{r}(n)\in\Lam$.
\begin{conjecture}
A set $\{p_{1},\ld,p_{r}\}$ of polynomial mappings $\Z^{m}\ra\Z^{k}$ has SPSZ property
iff the mappings $p_{1},\ld,p_{r}$ are jointly intersective.
\end{conjecture}

At this stage,
we are unable to check this conjecture by methods developed above
because of lack of theory of characteristic factors for $\Z^{k}$-actions,
similar to that established in \cite{HKo} and \cite{Zo} for $\Z$-actions.
%=====================================================Bibl

%=====================================================
\hbox to \hsize{\small\hfil
\vtop{\hsize=5cm
V. Bergelson\\
Department of Mathematics\\
The Ohio State University\\
Columbus, OH 43210, USA\\
{\it vitaly@math.ohio-state.edu}}
\hfil
\vtop{\hsize=5cm
A. Leibman\\
Department of Mathematics\\
The Ohio State University\\
Columbus, OH 43210, USA\\
{\it leibman@math.ohio-state.edu}}
\hfil}
\kern8mm
\hbox to \hsize{\small\hfil
\vtop{\hsize=8cm
E. Lesigne\\
Laboratoire de Math\'{e}matiques et Physique Th\'{e}orique\\
Universit\'{e} Francois-Rabelais Tours\\
F\'{e}d\'{e}ration Denis Poisson -- CNRS\\
Parc de Grandmont, 37200 Tours, France\\
{\it Emmanuel.Lesigne@lmpt.univ-tours.fr}}
\hfil}
%=====================================================
\end{document}